\documentclass[12pt]{article}
\author{MIESZKO BASZCZAK}
\title{\textbf{ACTION OF WEYL GROUP ON \\EQUIVARIANT K-THEORY OF \\FLAG VARIETIES}}
\date{\vspace{-2ex}}
\makeatletter
\catcode`! 3
\def\Transpose #1{\romannumeral0\expandafter
	\Mar@Transpose@a\romannumeral`^^@\Mar@DoOneRow #1\\!\\}

\def\Mar@DoOneRow #1\\{\Mar@DoOneRow@a {}#1&^^@&}%

\def\Mar@DoOneRow@a #1#2&{%
	\if^^@\detokenize{#2}\expandafter\@gobble\fi
	\Mar@DoOneRow@a {#1#2\\}%
}%

\def\Mar@Transpose@a #1#2\\{\ifx!#2\expandafter\Mar@FinishTranspose\fi
	\expandafter\Mar@Transpose@b\romannumeral`^^@\Mar@DoOneRow@a {}#2&^^@&#1}

\def\Mar@Transpose@b #1#2^^@\\{\Mar@Join {}#2^^@!#1}

\def\Mar@Join #1#2\\#3!#4\\%
{\if^^@\detokenize{#3}\expandafter\Mar@EndJoin\fi
	\Mar@Join {#1#2&#4\\}#3!}%

\def\Mar@EndJoin\Mar@Join #1^^@!^^@\\{\Mar@Transpose@a {#1^^@\\}}

\def\Mar@FinishTranspose
#1&^^@&#2\\^^@\\{ #2}

\catcode`! 12
\makeatother

\usepackage{fancyhdr}
\usepackage{amsmath}
\usepackage{array}
\usepackage{parskip}
\usepackage{iftex}
\ifPDFTeX
\usepackage[T1]{fontenc}
\usepackage{mathpazo}
\else
\usepackage{fontspec}
\fi
\usepackage{amsmath}
\usepackage{float}
\usepackage{amssymb,graphicx}
\usepackage{array}
\usepackage{mathtools}
\usepackage{amsthm}
\usepackage{xfrac}
\usepackage{stackrel}
\usepackage{multicol}
\usepackage{comment}
\usepackage{centernot}
\usepackage{xcolor}
\usepackage{graphicx}
\usepackage{mathrsfs}
\usepackage{lmodern}
\usepackage{subcaption}
\usepackage{tikz-cd}
\usepackage{hyperref}
\usepackage{xcolor}
\usepackage{lscape}
\usepackage{pdfpages}
\usepackage{amsfonts}
\usepackage{color}
\usepackage{colortbl}
\usepackage[backend=bibtex, natbib, style=alphabetic, maxbibnames=10, minalphanames=4, maxalphanames=4, doi=false,isbn=false,url=false,eprint=false]{biblatex}
\usepackage{geometry}

\newcommand\restr[2]{{
		\left.\kern-\nulldelimiterspace 
		#1 
		\vphantom{\big|}
		\right|_{#2}
}}

\definecolor{grey}{RGB}{220,220,220}

\theoremstyle{definition}
\newtheorem{defn}{Definition}[section]
\newtheorem{uwg}[defn]{Remark}

\theoremstyle{plain}

\newtheorem{tw}[defn]{Theorem}
\newtheorem{lem}[defn]{Lemma}

\newtheorem{stw}[defn]{Proposition}

\newcommand{\id}{\mathrm{id}}

\newcommand{\sgn}{\mathrm{sgn}}
\newcommand{\oo}[1][\omega]{[\mathcal{O}_{#1}]}
\newcommand{\ooo}[1][\omega]{[\mathcal{I}_{#1}]}
\newcommand{\oooo}[1][\omega]{[\mathcal{O}_{e_{#1}}]}
\newcommand{\xx}[1][\omega]{[X_{#1}]}
\newcommand{\llll}[1][\alpha_i]{[\mathcal{L}(#1)]}
\newcommand{\lllll}[1][-\alpha_i]{[\mathcal{L}(#1)]}
\newcommand{\dd}{\text{: }}
\newcommand{\mC}{\text{MC}}

\DeclareMathOperator{\Aut}{Aut}

\begin{document}
	\emergencystretch 3em
	
	\maketitle
	
	\begin{abstract}
		\noindent We describe the action of the Weyl group of a semi simple linear group $G$ on cohomological and K-theoretic invariants of the generalized flag variety $G/B$. We study the automorphism $s_i$, induced by the reflection in the simple root, on the equivariant $K$-theory ring $K_T(G/B)$ using divided difference operators. Using the localization theorem for torus action and Borel presentation for the equivariant K-theory ring, we calculate the formula for this automorphism. Moreover, we expand this formula in the basis consisting of structure sheaves classes of Schubert varieties. We provide effective formula (applying properties of Weyl groups) for the approximation of this expansion, more specifically for the part corresponding to Schubert varieties with the fixed dimension, which in the case of $G$ being a special linear group is more exact. Finally, we discuss the above-mentioned formula in the basis of motivic Chern classes of Schubert varieties.
	\end{abstract}

	\tableofcontents
	\section*{Introduction}
	Let $G$ be a complex semi-simple Lie group. Let $B$ be a Borel subgroup of $G$ and $G/B$ the associated full flag variety. The fundamental classes of Schubert varieties in $G/B$ form a basis of (equivariant) cohomology, treated as an algebra over the (equivariant) cohomology of a point. Similarly, the equivariant $K$-theory ring $K_T(G/B)$, described e.g. in \cite{Uma}, has a basis consisting of structure sheaves classes of Schubert varieties, treated as an algebra over $K_T(pt)$. \\ \\
	In both of these rings, there are many positivity properties. Such properties depict different invariants in different bases. From the most fundamental ones (\cite{Bri} in $K$-theory and \cite{AGM} in equivariant $K$-theory), through Chern-Schwartz-MacPherson classes positivity conjecture in ordinary cohomology (\cite{AMSS}) to the ones that are not proven yet, for example Chern-Schwartz-MacPherson classes positivity conjecture in equivairant cohomology (\cite{AM}) as well as motivic Chern classes positivity conjecture (\cite{FRW}), both strongly related to the automorphism $s_i$, which is induced by the reflection in a simple root of the Weyl group. Hence, an interesting question arises: how to describe this action on the basis of Schubert varieties? In cohomology, the answer is known and uses the Chevalley formula (as in \cite{AM}). In $K$-theory, however, the situation is a bit more complicated. \\ \\
	In this paper, we show that using the Chevalley formula one can calculate a part of the support of the automorphism $s_i$ of the Weyl group. Firstly, using the Borel presentation (Lemma \ref{Lemma - Borel presentation of cohomology} and Lemma \ref{Lemma - Borel presentation of K-theory}) we derive formulas
	\begin{align*}
		s_i^{coh} &= \id + c_1(\llll) \partial_i^{coh} \\
		s_i^K &= [\mathcal{L}(-\alpha_i)] \cdot \id - ([\mathcal{L}(-\alpha_i)] - 1)\partial_i^K
	\end{align*}
	for automorphism $s_i$ on both equivariant cohomology ring and equivariant $K$-theory ring (Lemma \ref{Lemma - formula for s_i^{coh}} and Lemma \ref{Lemma - formula for s_i^{K}}). Both of these formulas are well known facts, first of them is examined in \cite{Knu}, but the second one is not present anywhere in the currently available literature. The main result of this paper is Theorem \ref{Theorem - formula for s_i^K} proving the following formula
	\[
	s_i^K(\oo) = 
	\begin{cases*}
		a_{\omega s_i} [\mathcal{O}_{\omega s_i}]  +  \sum_{\beta \in C_i}  a_{\omega s_i s_\beta} [\mathcal{O}_{\omega s_i s_\beta}] + l.o.t. & \text{if } $l(\omega s_i) > l(\omega)$\\
		\oo & \text{otherwise}
	\end{cases*}
	\]
	where l.o.t. stands for lower order terms (see the explanation of the notion in Theorem \ref{Theorem - formula for s_i^K}). Moreover, for the special case of $SL_n/B$ we derive formula (Equation \eqref{Equation - s_i formula in SL_n})
	\begin{align*}
		s_i^K(\oo) &= (1 - e^{\omega (\alpha_i)}) \oo[\omega s_i] - \oo  \\ &- \sum_{\beta \neq \alpha_i \text{ positive root such that } l(\omega s_i s_\beta) = l(\omega) \text{ and } \langle \alpha_i, \beta \rangle > 0} e^{-\omega s_i s_\beta(\alpha_i)} \oo[\omega s_i s_\beta]  \\ &+ \sum_{\beta \neq \alpha_i \text{ positive root such that } l(\omega s_i s_\beta) = l(\omega) \text{ and } \langle \alpha_i, \beta \rangle < 0} e^{\omega(\alpha_i)} \oo[\omega s_i s_\beta] \\ &+ l.o.t.
	\end{align*}
	\\
	In Chapter \ref{preliminaries} we introduce the basic definitions and properties of the objects we study, such as equivariant cohomology and equivariant $K$-theory of flag varieties, classes of Schubert varieties in these rings and motivic Chern classes. In Chapter \ref{Action of Weyl group}, we describe the automorphism $s_i$ on equivariant cohomology and equivariant $K$-theory. Chapter \ref{Chevalley formula} clarifies the details of the Chevalley formula. Chapter \ref{Properties of action of Weyl group in oo basis} contains the proof of Theorem \ref{Theorem - formula for s_i^K} and Equation \eqref{Equation - s_i formula in SL_n}. In Chapter \ref{Properties of action of Weyl group in MC basis} we consider the automorphism $s_i$ in the motivic Chern classes basis and present the formula for this action in full generality as well as for some special cases. Chapter \ref{Conclusions} points out some problems and possible ways of further research in this field. \\ \\
	{\em Acknowledgements.} I would like to give a special thanks to Andrzej Weber for his support, guidance and valuable discussions, and Magdalena Zielenkiewicz for her remarks, comments and overall help in the editing process.
	\section{Preliminaries} \label{preliminaries}
	\subsection{Equivariant cohomology and equivariant $K$-theory} \label{preliminaries-1}
	Let $X$ be a quasi-projective complex algebraic variety. Let $H_*(Y)$ denote the Borel-Moore homology group and $H^*(X)$ the cohomology ring, both with rational coefficients. Every subvariety $Y \subseteq X$ of complex dimension $k$ defines \emph{a fundamental class} $[Y] \in H_{2k}(X)$ (\cite[Appendix B.3]{Ful}). For $X$ smooth we identify the Borel-Moore homology and cohomology via the Poincare duality. \\ \\
	Suppose a torus $T = (\mathbb{C}^*)^n$ acts on the variety $X$. \emph{The $T$-equivariant cohomology ring of $X$} is defined as the cohomology of the quotient $H_{T}^*(X) = H^*(ET \times^T X, \mathbb{Q})$ and $ET$ is the total space of the universal $T$-bundle. This ring is an algebra over the ring $H^*_T(pt)$, which is equal to
	\begin{equation*}
		H^*_T(pt) = H^*(ET/T, \mathbb{Q}) = H^*((\mathbb{P}^{\infty})^n, \mathbb{Q}) = \mathbb{Q}[t_1, \ldots, t_n].
	\end{equation*}
	Moreover, this ring can be identified with the polynomial ring $Sym_\mathbb{Q}\Xi(T)$ via the map $\lambda \to c_1(\mathcal{K}(\lambda))$, where $\Xi(T)$ is the group of characters of $T$ and $\mathcal{K}(\lambda)$ is the linear bundle $ET \times^T \mathbb{C} \to ET/T$. Given a $T$-invariant subvariety $Y$ we define its fundamental class $[Y]_T$ as the class of the subvariety $E_n \times^T Y$ of $E_n \times^T X$ for sufficiently large $n$, where $E_n$ is an approximation of the universal $T$-bundle. \\ \\
	We can consider three $K$-theories on $X$ - the $K$-theory of coherent equivariant sheaves, the $K$-theory of locally free equivariant sheaves and the $K$-theory of equivariant topological vector bundles. If $X$ is smooth, then the first and the second are isomorphic (\cite[Theorem 5.7]{Tho}). Moreover, if $X$ is a finite union of orbits, the natural map (forgetting the structure of a vector bundle) is an isomorphism between the second and the third (\cite[Theorem 9.1]{FRW}). In this case, we denote by $K_T(X)$ \emph{the equivariant $K$-theory ring of $X$}. $K_T(X)$ is an algebra over $K_T(pt)$, which is equal to $R(T) \simeq \mathbb{Z}[\Xi(T)]$.
	\subsection{Flag varieties, Schubert varieties and their classes} \label{preliminaries-2}
	Let $G$ be a connected, complex, semi-simple Lie group, $B_+$ a fixed Borel subgroup, $B_-$ the opposite Borel subgroup and $T \subseteq B_+$ a maximal torus. From now on, we denote $B_+$ by $B$. Let $W = N(T)/T$ be the Weyl group and $l\dd W \to \mathbb{N}$ the length function. Let $\alpha_1, \ldots, \alpha_n$ denote the associated set of simple roots. We denote by $s_i$ the reflection in the root $\alpha_i$ and by $P_i$ the minimal parabolic subgroup corresponding to the simple root $\alpha_i$. \\ \\
	The quotient $X := G/B$ is called \emph{the full flag variety} for $G$ and $B$ defined as above. The isomorphism $G/B \simeq G_{comp}/G_{comp}\cap T$, where $G_{comp}$ is a suitably chosen maximal compact subgroup of $G$, implies that the Weyl group acts on $G/B$ by right multiplication. The foundation of the Schubert calculus is a well known partition called the Bruhat decomposition (\cite[Section 14.12]{Bor}),
	\begin{equation*}
		G/B = \bigsqcup_{\omega \in W} B\tilde{\omega} B / B
	\end{equation*}
	for $\tilde{\omega}$ being a representative of $\omega \in NT/T$ in $NT$. The \emph{Schubert cell $X_\omega^\circ$} is defined as $B\tilde{\omega} B / B$. The \emph{Schubert varieties $X_\omega$} are defined as $\overline{B\tilde{\omega} B / B}$ (closure in Zariski topology). Because these varieties are $T$-invariant, one can take their fundamental classes $[X_\omega]_T$, and these classes form a $H_T^*(pt)$-basis of $H_T^*(G/B)$. The situation is similar when we study $K_T(G/B)$. The structure sheaves $\mathcal{O}_\omega$ of $X_\omega$ (which are $T$-equivariant) define classes $\oo$ in $K_T(G/B)$ that form a $K_T(pt)$-basis of $K_T(G/B)$.
	\subsection{Equivariant cohomology and equivariant $K$-theory of flag varieties} \label{preliminaries-3}
	\begin{defn}
		\emph{The Berstain-Gelfand-Gelfand operator $\partial_i^{coh}: H^*_T(G/B)  \to H^*_T(G/B)$} is defined as
		\begin{equation*}
			\partial_i^{coh}(u) = p_i^*p_{i*}(u),
		\end{equation*}
		where $p_i\dd G/B \to G/P_i$ is the projection and $u$ is an element of $H^*_T(G/B)$.
	\end{defn}
	The Demazure operators $\partial_i^K$ on $K_T(G/B)$ were defined in \cite{Dem}. They are analogs of the Bernstain-Gelfand-Gelfand operators.
	\begin{defn} \label{Definition - Demazure operator}
		\emph{The Demazure operator $\partial_i^K: K_T(G/B) \to K_T(G/B)$} is defined as
		\begin{equation*}
			\partial_i^K([\mathcal{U}]) = p_i^*p_{i*}([\mathcal{U}]),
		\end{equation*}
		where $p_i: G/B \to G/P_i$ is the projection and $[\mathcal{U}]$ is an element of $K_T(G/B)$.
	\end{defn}
	The Berstain-Gelfand-Gelfand operator and the Demazure operator satisfy the following.
	\begin{stw} [{\cite[Theorem 3.12]{BGG}}]  \label{Proposition - BGG operator}
		\begin{equation*} 
			\partial_i^{coh}([X_\omega]) = \begin{cases*}
				[X_{\omega s_i}]  & \text{if $l(\omega s_i) > l(\omega)$}, \\
				0 & \text{otherwise.}
			\end{cases*}
		\end{equation*}
	\end{stw} \begin{stw} [{\cite[Theorem 2]{Dem}}] \label{Proposition - Demazure operator}
		\begin{equation*}
			\partial_i^K([\mathcal{O}_\omega]) = \begin{cases*}
				[\mathcal{O}_{\omega s_i}]  & \text{if $l(\omega s_i) > l(\omega)$}, \\
				[\mathcal{O}_\omega] & \text{otherwise.}
			\end{cases*}
		\end{equation*}
	\end{stw}
	Let us denote the ring $H^*_T(pt)$ by $S$ and take an element $\lambda \in \Xi(T)$. This element can be extended to a character $\tilde{\lambda}$ of $B$ by composing it with the projection $B \to T$. One can take a line bundle $\mathcal{L}(\lambda) = G \times \mathbb{C}/B = G \times \mathbb{C}/(gb, z) \sim (g, \tilde{\lambda}(b) z)$ over $G/B$ which is $G$-equivariant, hence $T$-equivariant. Furthermore, there is a morphism 
	\begin{align*}
		\Xi(T) &\to H^*_T(X) \\
		\lambda &\to c_1^T(\mathcal{L}(\lambda)).
	\end{align*}
	This assignment induces a morphism of rings $c^T\dd S \to H_T^*(G/B)$.
	\begin{lem} \cite[Section 1.3]{AS} \label{Lemma - Borel presentation of cohomology}
		There is an isomorphism called the Borel presentation of $H^*_T(G/B)$
		\begin{equation*}
			S \otimes_{S^W} S \to H^*_T(G/B)
		\end{equation*}
		defined as $f \otimes g \to f \cdot c^T(g)$.
	\end{lem}
	The map $\gamma: R(T) \to K_T(G/B)$ induced by the assignment $\mathbb{Z} [\Xi(T)] \supseteq \Xi(T) \ni \lambda \to [\mathcal{L}(\lambda)] \in K_T(G/B)$, leads to the following lemma called \emph{the Borel's presentation} of $K_T(G/B)$.
	\begin{lem}  \cite[Theorem 4.4]{HLS} \label{Lemma - Borel presentation of K-theory}
		There is an isomorphism
		\begin{equation*}
			R(T) \otimes_{R(T)^W} R(T) \to K_T(G/B)
		\end{equation*}
		defined as $f \otimes g \to f \cdot \gamma(g)$.
	\end{lem}

	\subsection{Motivic Chern classes} \label{preliminaries-4}
	In this section, we restrict our attention only to equivariant motivic Chern classes of Schubert cells. For a general definition see e.g. \cite{FRW}. We define the equivariant motivic Chern class of Schubert cell $X_\omega^\circ$ in equivariant $K$-theory as follows.
	\begin{defn} \label{Definition - operator T_i}
		\emph{The equivariant motivic Chern class of Schubert's cell $X_\omega^\circ$} for the reduced presentation $s_{i_1} s_{i_2} \ldots s_{i_n}$ of $\omega$ is defined as
		\begin{equation*}
			\mC_y(X_\omega^\circ) = \mathcal{T}_{i_1} \mathcal{T}_{i_2} \ldots \mathcal{T}_{i_n} (\oo[\id]),
		\end{equation*}
		where \emph{the $K$-theoretic Demazure-Lusztig operator $\mathcal{T}_i$} on $K_T^*(G/B)[y]$ is defined as
		\begin{equation*}
			\mathcal{T}_i = \partial_i^K + y \llll \partial_i^K - \id.
		\end{equation*}
	\end{defn}
	\begin{uwg} [{\cite[Proposition 3.4, Remark 5.4]{AMSS2}}]
		The above defintion does not depend on the choice of presentation of element $\omega$.
	\end{uwg}
	The following equality holds (\cite[Corollary 5.2]{AMSS2}):
	\begin{equation*}
		\mathcal{T}_i(\mC_y(X_\omega^\circ)) = 
		\begin{dcases*}
			\mC_y(X_{\omega s_i}^\circ) & $\text{ if } l(\omega s_i) > l(\omega)$ \\
			-(y + 1) \mC_y(X_\omega^\circ) - y \mC_y(X_{\omega s_i}^\circ) & \text{ otherwise.}
		\end{dcases*}
	\end{equation*}
	These classes depend on $y$ and after setting $y = const$ form the basis of $K_T^*(G/B)$ localized in the ideal generated by the elements of the form $(1 + ye^{-\alpha_i})$. Two important special cases are $y = -1$ and $y = 0$. In the first case these classes form a \emph{fixed point basis} defined as $\oo[e_\omega]$ where $e_\omega := \widetilde{\omega} B \in (G/B)^T$ is the fixed point induced by the element $\omega \in W$ (\cite[Remark 5.3, Lemma 3.8]{AMSS2}). In the second case the class $\mC_0(X_\omega^\circ)$ is equal to \emph{the ideal sheaf of the boundary of the Schubert variety $X_\omega$ defined as $\mathcal{I}_\omega = \oo - \oo[\partial X_\omega]$} (\cite[Lemma 3.8]{AMSS2}). We will discuss the automorphism $s_i$ in these bases in Chapter \ref{Properties of action of Weyl group in MC basis}.
	\section{Action of Weyl group} \label{Action of Weyl group}
	\begin{uwg}
		Let us treat $\Xi(T)$ as a lattice in $\mathfrak{h}^*$, where $\mathfrak{h}$ is the Cartan subalgebra of the Lie algebra of $G$. By $(-,-)$ we denote an inner product induced by the Killing form. Moreover, we define \emph{the coroot $\beta^\vee$} of the root $\beta$ as $2 \frac{\beta}{(\beta, \beta)}$. We use the notation $\langle \alpha, \beta^\vee \rangle := 2 \frac{(\alpha, \beta)}{(\beta, \beta)}$.
	\end{uwg}
	For every element $\omega \in W$, we choose its representative $\widetilde{\omega}$ in $N(T)$. We define $m_\omega \in \Aut(G_{comp}/(G_{comp}\cap T))$ as a right hand side multiplication by $\widetilde{\omega}$. This automorphism, which is $T$-equivariant, combined with a homeomorphism $G_{comp}/(G_{comp}\cap T) \approx G/B$, induces (by pullback) automorphism $\omega^{coh} \in  \Aut(H^*_T(G/B))$ (\cite[Section 3.1]{MNS2}) and $\omega^K \in \Aut(K^*_T(G/B))$ (\cite[Section 5.1]{MNS2}). In next two sections, we will prove formulas for both of these actions for $\omega = s_i$.
	\begin{uwg}
		There are other actions of $W$ on the cohomology ring and $K$-theory ring. Instead of right hand side multiplication, one can take the left hand side multiplication. The induced action is completly described in \cite{MNS2}. Further, instead of taking pullback above, one can take the pushforward map, but we can easily switch from one case to another, because pullback is adjoint to pushforward. 
	\end{uwg}
	\subsection{Equivariant cohomology} \label{Action of Weyl group - 1}
	Let $\omega^S\dd S \to S$ be the automorphism acting on the generators of $S$ in the same manner as $\omega$ acts on the set of simple roots.
	\begin{defn} \label{Definition - divided difference operator}
		\emph{The divided difference operator $\widetilde{\partial}_i^{coh}\dd S \to S$} is defined as
		\begin{equation*}
			\widetilde{\partial}_i^{coh}(f) = \frac{f - s_i^S(f)}{-\alpha_i}.
		\end{equation*}
	\end{defn}
	\begin{uwg}
		This operator is sometimes defined as $\frac{f - s_i^S(f)}{\alpha_i}$, but for the flag variety $G/B_-$. By changing $\alpha_i$ for $-\alpha_i$, we translate this definition to our case.
	\end{uwg}
	\begin{uwg} \label{Remark - field of fraction}
		The expression $\frac{f - s_i^S(f)}{-\alpha_i}$ formally belongs to a field of fractions, but we know that $f - s_i^S(f)$ is divisible by $\alpha_i$.
	\end{uwg}
	The Bernstain-Gelfand-Gelfand operator is closely related to the divided difference operator. The relation between them is given by the isomorphism from Lemma \ref{Lemma - Borel presentation of cohomology} 
	\begin{equation*}
		\Phi: S \otimes _{S^W} S \to H_T^*(G/B), 
	\end{equation*} 
	which satisfies the following (\cite{Ara}, \cite{BGG} for non-equivariant case):
	\begin{equation*}
		\Phi \circ (\id \otimes \widetilde{\partial}_i^{coh}) = \partial_i^{coh} \circ \Phi.
	\end{equation*}
	Using that fact, we can rewrite the Definition \ref{Definition - divided difference operator}:
	\begin{align*}
		s_i^S &= \id + \alpha_i \widetilde{\partial}_i^{coh} \\
		\Phi \circ (\id \otimes s_i^S) &= \Phi + c_1(\llll) \partial_i^{coh} \circ \Phi \\
		\Phi \circ (\id \otimes s_i^S) \circ \Phi^{-1} &= \id + c_1(\llll) \partial_i^{coh}
	\end{align*}
	We give a topological proof of the following key formula (another proof of it is shown in \cite{Knu}).
	\begin{lem} \label{Lemma - formula for s_i^{coh}}
		$\Phi \circ (\id \otimes s_i^S) \circ \Phi^{-1} = s_i^{coh}$ in $\Aut(H_T^*(G/B))$.
	\end{lem}
	\begin{proof}
		The ring $H_T^*(G/B)$ can be treated as a subring of the localized ring $H_T^*(G/B)_{loc}$, which is defined as the localization of $H_T^*(G/B)$ in $\mathfrak{m}/\{0\}$ where $\mathfrak{m}$ is a maximal ideal of $H_T(pt)$ generated by the elements $t_1, t_2 \ldots t_n$. The fundamental classes $[x]$ of the $T$-fixed points of $G/B$ form a basis of \emph{the localized cohomology ring} (\cite[Section 2.2]{MNS2}), called \emph{the fixed point basis}. Let $\restr{\kappa}{x} = i_x^*(\kappa)$ denote the restriction of $\kappa$ to the fixed point $x$, where $\kappa \in H_T^*(G/B)$ and $i_x\dd \{x\} \to G/B$. We have the following commutative diagram:
		\begin{center}
		\begin{tikzcd}
			 S \otimes_{S^W} S \arrow[d, "\id \otimes s_i^S"] \arrow[r, "\Phi"] & H_T^*(G/B) \arrow[d, "\Phi \circ (\id \otimes s_i^S) \circ \Phi^{-1}"] \arrow[hookrightarrow]{r}  & H_T^*(G/B)_{loc} \arrow[d, "s_i"] \\
			S \otimes_{S^W} S  \arrow[r, "\Phi"] & H_T^*(G/B) \arrow[hookrightarrow]{r}  & H_T^*(G/B)_{loc}
		\end{tikzcd}%
		\end{center}
\ 
		\begin{center}
		\begin{tikzcd}
			1 \otimes \lambda \arrow[d, "1 \otimes s_i^S"] \arrow[r, "\Phi"] & c_1(\llll[\lambda]) \arrow[d, "\Phi \circ (\id \otimes s_i^S) \circ \Phi^{-1}"] \arrow{r}  & \{\omega(\lambda)\}_{\omega \in W}\arrow[d, "s_i"] \\
			1 \otimes s_i(\lambda)  \arrow[r, "\Phi"] & c_1(\llll[s_i(\lambda)]) \arrow{r}  & \{(\omega s_i)(\lambda)\}_{\omega \in W}
		\end{tikzcd}%
		\end{center}
		The right-most vertical arrow in this diagram is induced by topological action of Weyl group ($s_i$ permutes fixed points, so in cohomology it permutes summands of the direct sum) and so is the central arrow. This proves the lemma, so the formula
		\[ s_i^{coh} = \id + c_1(\llll) \partial_i^{coh}\]
		is satisfied.
	\end{proof}
	\subsection{Equivariant K-theory} \label{Action of Weyl group - 2}
	Defining the isobaric divided difference operator on the representation ring $R(T)$ is analogous to equivariant cohomology case. Let $\omega^R\dd R(T) \to R(T)$ be the automorphism acting on generators $e^\lambda$ of $R(T)$ by $\omega^R(e^\lambda) = e^{\omega^S(\lambda)}$.
	\begin{defn} [{\cite[Chapter 1]{HLS}}] \label{Definition - isobaric divided difference}
		\emph{The isobaric divided difference operator} $\widetilde{\partial}_i^{K}: R(T) \to R(T)$ is defined as
		\begin{equation*}
			\widetilde{\partial}_i^{K}(u) = \frac{u - e^{\alpha_i}s_i^R(u)}{1 - e^{\alpha_i}}.
		\end{equation*}
	\end{defn}
	\begin{uwg}
		In \cite{HLS} this operator is defined as $\widetilde{\partial}_i^{K}(u) = \frac{u - e^{-\alpha_i}s_i^R(u)}{1 - e^{-\alpha_i}}$, but for the flag variety $G/B_-$. By changing $-\alpha_i$ for $\alpha_i$, we translate this definition to our case.
	\end{uwg}
	\begin{uwg}
		As in Remark \ref{Remark - field of fraction}, the element $\frac{u - e^{-\alpha_i}s_i^R(u)}{1 - e^{-\alpha_i}}$ formally belongs to the field of fractions of $R(T)$, but since 
		\begin{equation*}
			\frac{e^\lambda - e^{\alpha_i}s_i^R(e^\lambda)}{1 - e^{\alpha_i}} = \frac{e^\lambda - e^{\alpha_i}e^{\lambda - \frac{2(\lambda, \alpha_i)}{(\alpha_i, \alpha_i)}\alpha_i}}{1 - e^{\alpha_i}} = e^{\lambda} \frac{1 - e^{\alpha_i(-\frac{2(\lambda, \alpha_i)}{(\alpha_i, \alpha_i)} + 1)} }{1 - e^{\alpha_i}}
		\end{equation*}
		it belongs to $R(T)$, because $\frac{2(\lambda, \alpha_i)}{(\alpha_i, \alpha_i)}$ is an integer.
	\end{uwg}
	Analogously to what we have already described in the cohomology case (Section \ref{Action of Weyl group - 1}), the isomorphism (Lemma \ref{Lemma - Borel presentation of K-theory}) 
	\begin{equation*}
		\Psi: R(T) \otimes_{R(T)^W} R(T) \to K_T(G/B)
	\end{equation*}
	results in the following lemma.
	\begin{lem}  [{\cite[Proposition 4.5]{HLS}}]
		The following equation is satisfied
		\begin{equation*}
			\partial_i^K \circ \Psi = \Psi \circ (\id \otimes \widetilde{\partial}_i^{K}).
		\end{equation*}
	\end{lem}
	Rewriting Definition \ref{Definition - isobaric divided difference} leads to the equation:
	\begin{align*}
		s_i^R &= e^{-\alpha_i} - (e^{-\alpha_i} - 1)\widetilde{\partial}_i^{K}\\
		\Psi(\id \otimes s_i^R)  &= \Psi(1 \otimes e^{-\alpha_i}) - (\Psi(1 \otimes e^{-\alpha_i}) - \Psi(1\otimes 1)) \cdot \partial_i^K \circ \Psi\\
		(\Psi(\id \otimes s_i^R)\Psi^{-1})([\mathcal{U}]) &= [\mathcal{L}(-\alpha_i)][\mathcal{U}] - ([\mathcal{L}(-\alpha_i)] - 1)\partial_i([\mathcal{U}])
	\end{align*}
	\begin{lem} \label{Lemma - formula for s_i^{K}}
		The equation $\Psi(\id \otimes s_i^R)\Psi^{-1} = s_i^{K}$ is satisfied in $\Aut(K_T(G/B))$.
	\end{lem}
	\begin{proof}
		Analogous to the proof of Lemma \ref{Lemma - formula for s_i^{coh}}. The formula
		\begin{equation*}
			s_i^K([\mathcal{U}]) = [\mathcal{L}(-\alpha_i)][\mathcal{U}] - ([\mathcal{L}(-\alpha_i)] - 1)\partial_i^K([\mathcal{U}])
		\end{equation*}
		is satisfied, where $[\mathcal{U}] \in K_T(G/B)$.
	\end{proof}
	\section{Chevalley formula} \label{Chevalley formula}
	Our next goal is to express the action of the Weyl group in basis of Schubert varieties, both in cohomology and $K$-theory. In both formulas for these actions (Lemma \ref{Lemma - formula for s_i^{coh}} and Lemma \ref{Lemma - formula for s_i^{K}}), multiplication by a linear bundle (respectively its Chern class) appears, which is described by \emph{the Chevalley formula}. This knowledge, along with Proposition \ref{Proposition - BGG operator} and Proposition \ref{Proposition - Demazure operator}, allows us to understand the way in which actions $s_i^{coh}$ and $s_i^K$ look like in the bases of Schubert varieties. \\ \\
	First, let us introduce the Chevalley formula in equivariant cohomology.
	\begin{tw} [{\cite[Theorem 11.1.7, part e]{Kum}}] \label{Theorem - Chevalley formula for cohomology}
		Let $[X_\omega]$ denote class of the Schubert variety and $\mathcal{L}(\alpha)$ a linear bundle for $\alpha \in \Xi(T)$. Then
		\begin{equation*}
			c_1(\mathcal{L}(\alpha)) [X_\omega] = \omega(\alpha) [X_\omega] - \sum_{\beta \text{ positive root such that } l(\omega s_\beta) = l(\omega)} \langle \alpha, \beta^\vee \rangle \xx[\omega s_\beta].
		\end{equation*}
	\end{tw}
	The presentation of the Chevalley formula for an equivariant $K$-theory requires more preparation. To start, we define a family of operators $T_\beta^0$ and $T_\beta^1$ on $R(T)$ for $\beta$ being a root.
	\begin{defn} \label{Definition - family of operators T}
		Let $\beta$ be a root.
		\begin{align*}
			T_\beta^1(e^\lambda) &= e^{s_\beta \lambda}; \\
			T_\beta^0(e^\lambda) &=
			\begin{dcases*}
				0 & \text{if } $\langle \lambda, \beta^\vee \rangle = 0$, \\
				e^\lambda + \ldots + e^{\lambda - (\lambda(\beta^\vee) - 1)\beta} & \text{if } $\langle \lambda, \beta^\vee \rangle > 0$, \\
				-e^{\lambda + \beta} - \ldots - e^{\lambda - \lambda(\beta^\vee)\beta} & \text{if } $\langle \lambda, \beta^\vee \rangle < 0$.
			\end{dcases*}
		\end{align*}
	\end{defn}
	We define a monoid $W'$ with the same set of generators and the same braid relations as a Weyl group $W$ but instead of relations $s_i^2 = 1$, relations $s_i^2 = s_i$ hold. There is an obvious map $\Delta$ from $W'$ to $W$, which for the element $x \in W'$ return the element of $W$ with the same reduced decomposition in $W$ as $x$ have in $W'$. For example $\Delta(s_i^2) = s_i$. Let $\omega$ have the reduced decomposition $s_{i_1} s_{i_2} \ldots s_{i_n}$ and $\epsilon \in \{0,1\}^n$. By $x(\epsilon, \omega)$ we denote element $\Delta(s_{i_{j_1}} \ldots s_{i_{j_k}})$ for $j_l$ such that $\epsilon_{j_l} = 1$. For example, if $\omega = s_1 s_2 s_3$ and $\epsilon = \{1,0,1\}$ we have $x(\epsilon, \omega) = s_1s_3$. The following Theorem describes the action of left multiplication by $[\mathcal{L}(\alpha)]$.
	\begin{tw} [{\cite[Chapter 1]{Wil}}] \label{Theorem - Chevalley formula for K-theory}
		Let $\omega$ have the reduced decomposition $s_{i_1} s_{i_2} \ldots s_{i_n}$. Then, for every $\alpha \in \Xi(T)$ the formula
		\begin{equation*}
			[\mathcal{L}(\alpha)] \oo = \sum_{(\epsilon_1, \epsilon_2, \ldots, \epsilon_n) \in \{0,1\}^n} T_{-\alpha_{i_1}}^{\epsilon_1} \ldots T_{-\alpha_{i_n}}^{\epsilon_n} (e^\alpha) \oo[x(\epsilon, \omega)]
		\end{equation*}
		holds.
	\end{tw}
	\begin{uwg}
		In \cite{Wil}, the Chevalley formula is considered on flag variety $G/B_-$ and that is the reason why operators $T_{-\alpha_i}$ appear in this formula instead of operators $T_{\alpha_i}$.
	\end{uwg}
	\begin{uwg}
		To check the correctness of the above-mentioned remark let us examine an example of a flag variety $SL_2/B$. It is exactly the projective space $\mathbb{P}^1$. It consists of two Schubert cells - a point and a one-dimensional disc. Both of them correspond to the fixed points of the action of the Weyl group. Let us name these points $0$ and $\infty$. We know that
		\begin{align*}
			&\restr{\oo[\id]}{0} = 1 - e^{\alpha_1} &\restr{\oo[\id]}{\infty} = 0 \\
			&\restr{\oo[s_1]}{0} = 1 &\restr{\oo[s_1]}{\infty} = 1\\
			&\restr{\llll[\alpha_1]\oo[\id]}{0} = e^{\alpha_1} - e^{2\alpha_1} &\restr{\llll[\alpha_1]\oo[\id]}{\infty} = 0\\
			&\restr{\llll[\alpha_1]\oo[s_1]}{0} = e^{\alpha_1} &\restr{\llll[\alpha_1]\oo[s_1]}{\infty} = e^{-\alpha_1}\\
		\end{align*}
		Classes in the equivariant $K$-theory are determined by their restricions to the fixed point basis. That gives us equalities
		\begin{align*}
			\llll[\alpha_1]\oo[\id] &= e^{\alpha_1} \oo[\id] \\
			\llll[\alpha_1]\oo[s_1] &= e^{-\alpha_1} \oo[s_1] - (1 + e^{-\alpha_1}) \oo[\id].
		\end{align*}
		These calculations coincide with the formulas from Theorem \ref{Theorem - Chevalley formula for K-theory}.
	\end{uwg}
	\section{Properties of action of Weyl group in $\oo$ basis} \label{Properties of action of Weyl group in oo basis}
	Our goal is to study the $s_i^K$ action in the basis $\oo$ in equivariant $K$-theory. From now on, this basis will be denoted by $\mathcal{B}$. 
	\subsection{Proof of the main theorem} \label{Properties of action of Weyl group in oo basis - 1}
	In this subsection we prove the following theorem.
	\begin{tw} \label{Theorem - formula for s_i^K}
		Let $\omega \in W$. Then
		\[
			s_i^K(\oo) = 
			\begin{cases*}
				a_{\omega s_i} [\mathcal{O}_{\omega s_i}]  +  \sum_{\beta \in C_i}  a_{\omega s_i s_\beta} [\mathcal{O}_{\omega s_i s_\beta}] + l.o.t. & \text{if } $l(\omega s_i) > l(\omega)$\\
				\oo & \text{otherwise}
			\end{cases*}
		\]
	where $C_i = \{\gamma\dd \gamma \neq \alpha_i \text{ is a positive root such that } l(\omega s_i s_\gamma) = l(\omega) \text{ and } \langle \alpha_i, \gamma \rangle \neq 0\}$, $a_{\omega'}$ are non-zero elements of $R(T)$ and $l.o.t.$ is the sum of terms $a_{\omega'} [\mathcal{O}_{\omega'}]$ for $l(\omega') < l(\omega)$.
	\end{tw}
	\emph{The support} of action for fixed $\omega$ is the set of elements $\oo[\omega']$, that in formula for this action in the basis $\mathcal{B}$ have non-zero coefficient $a_{\omega'}$. Part of the support, which consist of elements $\oo[\omega']$ such that $l(\omega') = n$ for $n \in \mathbb{N}$ is called \emph{the support of lenght $n$} of this action.
	\subsubsection{Part I - preliminaries} \label{Properties of action of Weyl group in oo basis - 1 - 1}
	Lemma \ref{Lemma - formula for s_i^{K}} and Proposition \ref{Proposition - Demazure operator} imply that
	\begin{equation} \label{Equation - s_i^K formula}
		s_i^K(\oo) = \lllll \oo - (\lllll - 1)\oo[\omega s_i]     
	\end{equation}
	for $l(\omega s_i) > l(\omega)$ and that
	\begin{equation*}
		s_i^K(\oo) = \lllll \oo - (\lllll - 1)\oo = \oo
	\end{equation*}
	for $l(\omega s_i) < l(\omega)$. The second equation proves case $l(\omega s_i) \leq l(\omega)$ of Theorem \ref{Theorem - formula for s_i^K}. From now on we consider the first case only. The first step in our proof is simplifying Equation \ref{Equation - s_i^K formula} using the Chevalley formula. It is obvious that
	\begin{equation*}
		[\mathcal{L}(-\alpha_i)] [\mathcal{O}_\omega] = e^{-\omega(\alpha_i)} [\mathcal{O}_\omega] + l.o.t.
	\end{equation*}
	and
	\begin{align*}
		\lllll \oo[\omega s_i] &= e^{\omega s_i(-\alpha_i)} \oo[\omega s_i] + \text{ terms of length }l(\omega) + l.o.t. \\&= e^{\omega (\alpha_i)} \oo[\omega s_i] + \text{ terms of length }l(\omega) + l.o.t.    
	\end{align*}
	where "terms of length $l(\omega)$" is a linear combination (over $R(T)$) of the elements from the support of length $l(\omega)$ of left-multiplication by $\lllll$. Applying above equations to Equation \eqref{Equation - s_i^K formula} we obtain:
	\begin{equation} \label{Equation - s_i^K terms}
		s_i^K(\oo) = (1 - e^{\omega (\alpha_i)}) \oo[\omega s_i] + e^{-\omega(\alpha_i)} \oo - \text{ terms of length }l(\omega) + l.o.t.
	\end{equation}
	What needs to be understood is which elements belong to the support of length $l(\omega)$ of left-multiplication by $\lllll$. 
	\subsubsection{Part II - properties of Weyl group} \label{Properties of action of Weyl group in oo basis - 1 - 2}
	Let us fix a reduced decomposition $s_{i_1} s_{i_2} \ldots s_{i_n}$ of an element $\omega$ of the Weyl group. The Chevalley formula implies that elements $\oo[\omega']$ in the support of length $l(\omega) - 1$ of left-multiplication by $\lllll$ are in correspondence with the subwords $s_{i_1} s_{i_2} \ldots \widehat{s_{i_j}} \ldots s_{i_n}$ of $\omega$. Moreover, the following lemma holds.
	\begin{lem} \label{Lemma - removing elements}
		Let $s_{i_1} s_{i_2} \ldots s_{i_n}$ be a fixed representation of the word $\omega$. Then the words $\omega_k = s_{i_1} s_{i_2} \ldots \widehat{s_{i_k}} \ldots s_{i_n}$ and $\omega_l = s_{i_1} s_{i_2} \ldots \widehat{s_{i_l}} \ldots s_{i_n}$ are different.
	\end{lem}
	\begin{proof}
		Assume that $k = 1$ and $l = n$. Otherwise, we can consider the word $s_{i_{k - 1}} \ldots s_{i_1} \omega s_{i_n} \ldots s_{i_{l + 1}}$. The following equalities hold:
		\begin{align*}
			\omega_k &= \omega_l \\
			s_{i_1} \omega_k &= s_{i_1} \omega_l \\
			\omega &= \omega'
		\end{align*} 
		where $\omega'$ is $\omega$ with the first and the last element removed. Nevertheless $l(\omega') = l(\omega) - 2 \neq l(\omega)$,  which is a contradiction.
	\end{proof}
	The next step is to prove the following lemma.
	\begin{lem} \label{Lemma - correspodance removing one element}
		Let $\omega = s_{i_1} s_{i_2} \ldots s_{i_n}$. The subwords of $\omega$ of length $l(\omega) - 1$ obtained by removing one element from $\omega$ are in bijection with the words $\omega s_\beta$ such that $l(\omega s_\beta) = l(\omega) - 1$ for $\beta$ being a positive root.
	\end{lem}
	\begin{proof}
		\ 
		\begin{itemize}
			\item [$(\Leftarrow)$]
			\begin{lem}{\cite[Theorem 5.8]{Hum}} [Strong Exchange Condition]
				If $\omega = s_{i_1} s_{i_2} \ldots s_{i_n}$ is a reduced word and $t$ is a reflection satisfying $l(\omega t) < l(\omega)$ then there exist $k$ such that $\omega t =  s_{i_1} s_{i_2} \ldots \widehat{s_{i_k}} \ldots s_{i_n}$. 
			\end{lem}
			\item [$(\Rightarrow)$]
			We must prove that $s_{i_1} \ldots \widehat{s_{i_k}} \ldots s_{i_n}$ is equal to $\omega t$ for reflection in the positive root $\beta$. Let $\beta = \omega'(\alpha_k)$ where $\omega'$ is defined as $s_{i_n} \ldots s_{i_{k+1}}$.
			\begin{lem}
				The root $\beta$ defined as above is a positive root.
			\end{lem}
			\begin{proof}
				From \cite[Theorem 5.4]{Hum} it follows that if $l(\omega s_\alpha) > l(\omega)$ for a simple root $\alpha$, then $\omega(\alpha)$ is a positive root. We know that $\omega' s_{i_{k}}$ is a reduced word, because $\omega$ is as well, so $l(\omega' s_{i_{k}}) > l(\omega')$, so $\omega'(\alpha_k)$ is a positive root.
			\end{proof}
			Moreover, the element $\omega' s_{\alpha_k} (\omega')^{-1}$ is a reflection in a positive root $\beta$ (\cite[Lemma 5.7]{Hum}), so when we take $t = s_\beta$, the desired equality holds:
			\begin{equation*}
				\omega t = \omega s_\beta = \omega \omega' s_{i_k} (\omega')^{-1} = s_{i_1} \ldots \widehat{s_{i_k}} \ldots s_{i_n}.
			\end{equation*}
			\end{itemize}
	\end{proof}
	\subsubsection{Part III - support of Chevalley formula} \label{Properties of action of Weyl group in oo basis - 1 - 3}
	Let $\sgn\dd \mathbb{Z} \to \{+,-,0\}$ denote the function of the sign of a number. The following lemma holds.
	\begin{lem} \label{Lemma - equality of products}
		For $\beta$, $\omega'$ and $\alpha_k$ defined as above and for $\alpha$ a root, the following equality holds
		\begin{equation*}
			\sgn \langle (\omega')^{-1}(\alpha), \alpha_k^\vee \rangle = \sgn ( \beta, \alpha ). 
		\end{equation*}
	\end{lem}
	\begin{proof}
		\begin{align*}
			\sgn \langle (\omega')^{-1}(\alpha), \alpha_k^\vee \rangle &= \sgn \frac{2 ((\omega')^{-1}(\alpha), \alpha_k)}{(\alpha_k, \alpha_k)} = \sgn (2 ((\omega')^{-1}(\alpha), \alpha_k)) =\\ &=  \sgn ((\alpha, \omega'(\alpha_k))) = \sgn ((\omega'(\alpha_k), \alpha)) = \sgn (\beta, \alpha)
		\end{align*}
	\end{proof}
	From Definition \ref{Definition - family of operators T} we know, that for a fixed $\omega$ and a fixed $i$ the support of length $l(\omega)$ of left-multiplication $\oo[\omega s_i]$ by $\lllll$ consist of elements $\oo[\omega'']$ such that $\omega'' = s_{i_1} \ldots \widehat{s_{i_k}} \ldots s_{i_n} s_i$ and $\langle (\omega')^{-1}(\alpha_i),\alpha_k^\vee \rangle \neq 0$ for $\omega'$ defined as in the Subsection \ref{Properties of action of Weyl group in oo basis - 1 - 2}. By Lemma \ref{Lemma - removing elements} this is equivalent to $\omega'' = \omega s_i s_\beta$ for positive roots $\beta$ satisfying $(\beta, \alpha_i) \neq 0$ and $l(\omega s_i s_\beta) = l(\omega)$. The only detail demanding explanation now is that the coefficient of $\oo$ in $\lllll \oo[\omega s_i]$ is not equal to $e^{-\omega(\alpha_i)}$. If that was a case, the support of $s_i^K(\oo)$ would not contain $\oo$ from Equation \eqref{Equation - s_i^K terms} (because $\lllll \oo = e^{-\omega(\alpha_i)} \oo + l.o.t.$). The Chevalley formula implies that this coefficient is equal to $e^{-\omega(\alpha_i)} + 1$. Therefore, Theorem \ref{Theorem - formula for s_i^K} is proved.
	\subsection{Example of $SL_n/B$} \label{Properties of action of Weyl group in oo basis - 2}
	In this section, the main goal is to study the automorphism $s_i^K$ on $K_T(SL_n/B)$. In this example, we can calculate coefficients $a_{\omega s_i s_\beta}$ from Theorem \ref{Theorem - formula for s_i^K}. \\ \\ Every two different roots $\alpha, \beta$  of the Lie algebra $\mathfrak{sl}_n$ satisfy $(\alpha, \beta) \in \{-1,0,1\}$, so the operator $T^0_\beta$ in the Chevalley formula (Definition \ref{Definition - family of operators T}) simplifies to
	\begin{align} \label{Equation - operators T for SL_n}
		T_\beta^0(e^\lambda) &=
		\begin{dcases*}
			0 & \text{if } $\langle \lambda,\beta^\vee \rangle = 0$, \\
			e^\lambda & \text{if } $\langle \lambda,\beta^\vee \rangle = 1$, \\
			e^\lambda + 1 & \text{if } $\lambda = \beta$  \\
			-1 - e^{-\lambda} & \text{if } $\lambda = -\beta$  \\
			-e^{\lambda + \beta}& \text{if } $\langle \lambda,\beta^\vee \rangle = -1$.
		\end{dcases*}
	\end{align}
	First, we calculate the coefficient of the element $\oo$ in $\lllll \oo[\omega s_i]$. Equation \eqref{Equation - operators T for SL_n} and Lemma \ref{Lemma - removing elements} show that this coefficient is equal to $1 + e^{-\omega(\alpha_i)}$. Equation \eqref{Equation - s_i^K terms} gives
	\begin{equation*}
		s_i^K(\oo) = e^{-\omega(\alpha_i)} \oo - (1 + e^{-\omega(\alpha_i)}) \oo + \ldots = -\oo + \ldots.
	\end{equation*}
	Since we know $a_\omega = -1$, we calculate the coefficient $a_{\omega s_i s_\beta}$ of the element $\oo[\omega s_i s_\beta]$ in $\lllll \oo[\omega s_i]$. Let us assume that $\omega s_i s_\beta$ is the subword of $\omega s_i =  s_{i_1} \ldots \ldots s_{i_n} s_i$ with one element removed as follows: $s_{i_1} \ldots \widehat{s_{i_k}} \ldots s_{i_n} s_i$. Let $\lambda = s_{i_{k + 1}} \ldots s_i(-\alpha_i)$. By the Chevalley formula, $a_{\omega s_i s_\beta}$ is equal to 
	\begin{align*}
		T_{-\alpha_{i_1}}^1 \ldots T_{-\alpha_{i_k}}^0 \ldots T_{-\alpha_{i_n}}^1 T_{-\alpha_i}^1(e^{-\alpha_i}) &= T_{-\alpha_{i_1}}^1 \ldots T_{-\alpha_{i_k}}^0 (e^{s_{i_{k + 1}} \ldots s_i(-\alpha_i)}) = \\ &= T_{-\alpha_{i_1}}^1 \ldots T_{-\alpha_{i_k}}^0(e^{\lambda}).
	\end{align*}
	We will prove, that $\lambda \neq \pm \alpha_{i_k}$. Firstly, let us assume that $i_k \neq i$. By \cite[Theorem 5.4]{Hum}  if $l(\omega s_\alpha) > l(\omega)$ for a simple root $\alpha$, then $\omega(\alpha)$ is a positive root. Applying this to $\omega' = s_{i_{k + 1}} \ldots s_{i_n}$ and $\omega'' = s_{i_k} \ldots s_{i_n}$ implies that both $\lambda$ and $-\lambda$ are simple roots, which is a contradiction. If $i_k = i$, then Lemma \ref{Lemma - removing elements} implies that $s_{i_1} \ldots \widehat{s_{i_k}} \ldots s_{i_n} s_i = \omega$ and this case is already calculated. From Equation \eqref{Equation - operators T for SL_n}, three options arise.
	\begin{itemize}
		\item $\langle \lambda, -\alpha_k^\vee \rangle = 1$. From Equation \eqref{Equation - operators T for SL_n}
		\begin{align*}
			a_{\omega s_i s_\beta} &= T_{-\alpha_{i_1}}^1 \ldots T_{-\alpha_{i_k}}^0(e^{\lambda}) = T_{-\alpha_{i_1}}^1 \ldots T_{-\alpha_{i_{k-1}}}^1(e^{\lambda}) = \\ &= T_{-\alpha_{i_1}}^1 \ldots T_{-\alpha_{i_{k-1}}}^1(e^{s_{i_{k + 1}} \ldots s_i(-\alpha_i)}) = e^{-\omega s_i s_\beta(\alpha_i)}.
		\end{align*}
		\item $\langle \lambda, -\alpha_k^\vee \rangle = 0$, which implies that $T_{-\alpha_{i_k}}^0(e^\lambda)$ is equal to zero and therefore $a_{\omega s_i s_\beta} = 0$.
		\item $\langle \lambda, -\alpha_k^\vee \rangle = -1$. Equation \eqref{Equation - operators T for SL_n} gives
		\begin{equation*}
			T_{-\alpha_{i_k}}^0(e^\lambda) = -e^{\lambda - \alpha_k}
		\end{equation*}
		Now, let us establish the conditions for $\lambda$ to satisfy the equation $\langle \lambda, -\alpha_k^\vee \rangle  = -1$. In $\mathfrak{sl_n}$, presenting an element as the sum of simple roots makes every coefficient equal to either $-1$, $0$ or $1$.  Moreover, if coefficients of $\alpha_{i - 1}$ and $\alpha_{i + 1}$ are non-zero, then so is the coefficient of $\alpha_i$ and all of them are equal. If we want the equality $\langle \lambda, -\alpha_k^\vee \rangle  = -1$ to hold, then $\lambda$ written as a sum of simple roots must look as follows: $\alpha_i + \ldots + \alpha_{k-1} + \alpha_{k}$ or $\alpha_{k} + \alpha_{k+1} + \ldots + \alpha_j$. In both cases $s_k(\lambda) = \lambda - \alpha_k$, because
		\begin{equation*}
			s_i(\alpha_j) = 
			\begin{dcases*}
				\alpha_i + \alpha_j & \text{ if } $i = j \pm 1$ \\
				-\alpha_j & \text{ if } $i = j$ \\
				\alpha_j & \text{ otherwise.} \\
			\end{dcases*}
		\end{equation*}
		That gives us the following
		\begin{equation*}
			T_{-\alpha_{i_k}}^0(e^\lambda) = -e^{\lambda - \alpha_k} = -e^{s_k(\lambda)}.
		\end{equation*}
		It follows that 
		\begin{align*}
			a_{\omega s_i s_\beta} &= T_{-\alpha_{i_1}}^1 \ldots T_{-\alpha_{i_k}}^0(e^{\lambda}) = T_{-\alpha_{i_1}}^1 \ldots T_{-\alpha_{i_{k - 1}}}^1(-e^{s_k(\lambda)}) \\ &= T_{-\alpha_{i_1}}^1 \ldots T_{-\alpha_{i_{k - 1}}}^1(-e^{s_k(s_{i_{k + 1}} \ldots s_i(-\alpha_i))}) = -e^{\omega s_i(-\alpha_i)} = -e^{\omega (\alpha_i)}.
		\end{align*}
	\end{itemize}
	Combining these results with the Equation \eqref{Equation - s_i^K terms} and Lemma \ref{Lemma - equality of products}, we obtain:
	\begin{align} \label{Equation - s_i formula in SL_n}
		s_i^K(\oo) &= (1 - e^{\omega (\alpha_i)}) \oo[\omega s_i] - \oo  \nonumber \\ &- \sum_{\beta \neq \alpha_i \text{ positive root such that } l(\omega s_i s_\beta) = l(\omega) \text{ and } \langle \alpha_i, \beta \rangle > 0} e^{-\omega s_i s_\beta(\alpha_i)} \oo[\omega s_i s_\beta]  \nonumber \\ &+ \sum_{\beta \neq \alpha_i \ \text{ positive root such that } l(\omega s_i s_\beta) = l(\omega) \text{ and } \langle \alpha_i, \beta \rangle < 0} e^{\omega(\alpha_i)} \oo[\omega s_i s_\beta] \nonumber \\ &+ l.o.t.
	\end{align} 
	As an example of this formula, we are presenting the full formula for the $s_1$ action below, in the case of the flag variety $SL_3/B$.
	\begin{figure}[H]
		\begin{center}
			\begin{tabular}{c|cccccc}
			&$\oo[\id] $&$ \oo[s_1] $&$ \oo[s_2] $&$ \oo[s_1s_2] $&$ \oo[s_2s_1] $&$ \oo[s_1s_2s_1]$\\
			\hline
			$\oo[\id] $&$-1 $&$ 0 $&$ 0 $&$ 0 $&$ 0 $&$ 0$ \\
			$\oo[s_1] $&$1-e^{\alpha _1} $&$ 1 $&$ -e^{\alpha _1} $&$ 1 $&$ 0 $&$ 0$ \\
			$\oo[s_2] $&$0 $&$ 0 $&$ -1 $&$ 0 $&$ 0 $&$ 0$ \\
			$\oo[s_1s_2] $&$0 $&$ 0 $&$ 0 $&$ -1 $&$ 0 $&$ 0$ \\
			$\oo[s_2s_1] $&$0 $&$ 0 $&$ 1-e^{\alpha _1+\alpha _2} $&$ e^{\alpha _2} $&$ 1 $&$ 0$ \\
			$\oo[s_1s_2s_1] $&$0 $&$ 0 $&$ 0 $&$ 1-e^{\alpha _2} $&$ 0 $&$ 1$ \\
		\end{tabular}
		\end{center}
		\caption{Matrix of automorphism $s_1$ in the basis $\mathcal{B}$ for the flag variety $SL_3/B$.} \label{Figure - action s_1 in SL_3} 
	\end{figure}
\section{Properties of action of Weyl group in $\mC_y(X_\omega^\circ)$ basis} \label{Properties of action of Weyl group in MC basis}
In this chapter we study the automorphism $s_i^K$ in the $mC_y(X_\omega^\circ)$ basis of $K_T(G/B)_{loc}$. Using Lemma \ref{Lemma - formula for s_i^{K}}.
\begin{align*} 
	s_i^K(mC_y(X_\omega^\circ)) &= \lllll mC_y(X_\omega^\circ) - (\lllll - 1) \partial_i^K(mC_y(X_\omega^\circ))
\end{align*}
Since $\llll \lllll = \id$ (because $e^{\alpha_i} \cdot e^{-\alpha_i} = 1$), we get
\begin{align} \label{Equation - s_i formula in MC basis}
	s_i^K(mC_y(X_\omega^\circ)) &= \lllll \left(mC_y(X_\omega^\circ) - (1 - \llll) \partial_i^K(mC_y(X_\omega^\circ))\right).
\end{align}
From Definition \ref{Definition - operator T_i}, the following holds for $y \neq 0$:
\begin{equation*}
	\llll \partial_i^K = \frac{\mathcal{T}_i - \partial_i^K + \id}{y}.
\end{equation*}
That leads to the equation
\begin{align*}
	&s_i^K(mC_y(X_\omega^\circ)) = \\ &= \lllll \left(mC_y(X_\omega^\circ) - \partial_i^K(mC_y(X_\omega^\circ)) + \frac{\mathcal{T}_i - \partial_i^K + \id}{y}(mC_y(X_\omega^\circ)) \right) = \\ &= \lllll \frac{ mC_y(X_{\omega s_i}^\circ) - (y + 1)\partial_i^K(mC_y(X_\omega^\circ)) + (y + 1)mC_y(X_\omega^\circ)}{y}.
\end{align*}
Unfortunately, in this basis no easy description of the action $\partial_i^K$ is known, which makes such presentation unuseful. For $y = -1$ the situation is much easier as we will see further in this chapter.\\ \\
Another basis worth considering is the basis of motivic Chern classes with $y = 0$, which is the basis of classes $\mathcal{I}_\omega = \oo - \oo[\partial X_\omega]$. They satisfy the following equation
\begin{equation*}
	(\partial_i^K - \id)(\mathcal{I}_\omega) = \mathcal{T}_i(\mathcal{I}_\omega) = 
	\begin{dcases*}
		\mathcal{I}_{\omega s_i} & $\text{ if } l(\omega s_i) > l(\omega)$ \\
		- \mathcal{I}_\omega & \text{ otherwise.}
	\end{dcases*}
\end{equation*} 
Therefore
\begin{equation*}
	s_i^K(\mathcal{I}_\omega) =
	\begin{dcases*}
		(\id - \lllll) \mathcal{I}_{\omega s_i} + \mathcal{I}_\omega & $\text{ if } l(\omega s_i) > l(\omega)$ \\
		\lllll \mathcal{I}_\omega & \text{ otherwise.}
	\end{dcases*}
\end{equation*}
As we have just shown, to describe the automorphism $s_i^K$ in this basis, we need to use the Chevalley formula only once, not - as in the basis $\mathcal{B}$ - twice. The description of the Chevalley formula in the $\mathcal{I}_\omega$ basis demands further calculations. 
\begin{figure}[H]
	\centering{
	\begin{tabular}{c|cccccc}
		&$\ooo[\id] $&$ \ooo[s_1] $&$ \ooo[s_2] $&$ \ooo[s_1s_2] $&$ \ooo[s_2s_1] $&$ \ooo[s_1s_2s_1]$\\
		\hline
	\Transpose{
			$\ooo[\id] $&$ \ooo[s_1] $&$ \ooo[s_2] $&$ \ooo[s_1s_2] $&$ \ooo[s_2s_1] $&$ \ooo[s_1s_2s_1]$\\
			$-e^{\alpha _1} $&$ 1-e^{\alpha _1} $&$ 0 $&$ 0 $&$ 0 $&$ 0$ \\
			$1 + e^{\alpha_1} $&$ e^{\alpha _1} $&$ 0 $&$ 0 $&$ 0 $&$ 0$ \\
			$-e^{\alpha _1+\alpha _2} $&$ -e^{\alpha _1+\alpha _2} $&$ -e^{\alpha _1+\alpha _2} $&$ 0 $&$ 1-e^{\alpha _1+\alpha _2} $&$ 0$ \\
			$e^{\alpha _1+\alpha _2} $&$ e^{\alpha _1+\alpha _2} $&$ e^{\alpha _1+\alpha _2} $&$ -e^{\alpha _2} $&$ e^{\alpha _1+\alpha _2} $&$ 1-e^{\alpha_2}$ \\
			$e^{\alpha _1+\alpha _2} $&$ e^{\alpha _1+\alpha _2} $&$ 1 + e^{\alpha_1+\alpha_2} $&$ 0 $&$ e^{\alpha _1+\alpha _2} $&$ 0$ \\
			$-e^{\alpha _1+\alpha _2} $&$ -e^{\alpha _1+\alpha _2} $&$ -e^{\alpha _1+\alpha _2} $&$ 1 + e^{\alpha_2} $&$ -e^{\alpha_1+\alpha _2} $&$ e^{\alpha _2}$ 
}%
	\end{tabular}
	\caption{Matrix of automorphism $s_1$ in the $\ooo$ basis for the flag variety $SL_3/B$.}
}
\end{figure}
The simplest case is the basis of motivic Chern classes with fixed $y = -1$. In such basis, the automorphism $s_i$ is well described. The following equality holds (\cite[Lemma 3.7]{AMSS2}):
\begin{equation} \label{Equation - Chevalley formula in basis of fixed points}
	\lllll \oo[e_\omega] = e^{-\omega \alpha_i} \oo[e_\omega].	
\end{equation}
Therefore, Definition \ref{Definition - operator T_i} takes the following form 
\begin{equation*}
	\mathcal{T}_i^K = \partial_i^K - \llll \partial_i^K - \id
\end{equation*}
and Equation \eqref{Equation - s_i formula in MC basis} presents as
\begin{equation*}
	s_i^K([\mathcal{U}]) = \lllll(-\mathcal{T}_i([\mathcal{U}])).
\end{equation*}
Substituting $\oo[e_\omega]$ in the latter leads to the following formula
\begin{equation*}
	s_i^K(\oo[e_\omega]) = -\lllll(\mathcal{T}_i(\oo[e_\omega])) = -\lllll \oo[e_{\omega s_i}]
\end{equation*}
which, combined with the Equation \eqref{Equation - Chevalley formula in basis of fixed points}, results in the final formula
\begin{equation*}
	s_i^K(\oo[e_\omega]) = -e^{\omega \alpha_i} \oo[e_{\omega s_i}].
\end{equation*}
Unfortunately, topological properties of this basis are not interesting, because elements of it are defined as the classes of the sheaf over point, which does not provide much information about the topology of the flag variety. However, we can write the action of the Weyl group in it explicitly, which in itself is interesting.
\begin{figure}
	\centering{
		\begin{tabular}{c|cccccc}
			&$\oooo[\id] $&$ \oooo[s_1] $&$ \oooo[s_2] $&$ \oooo[s_1s_2] $&$ \oooo[s_2s_1] $&$ \oooo[s_1s_2s_1]$\\
			\hline
			\Transpose{
				$\oooo[\id] $&$ \oooo[s_1] $&$ \oooo[s_2] $&$ \oooo[s_1s_2] $&$ \oooo[s_2s_1] $&$ \oooo[s_1s_2s_1]$\\
				$0 $&$ -e^{\alpha _1} $&$ 0 $&$ 0 $&$ 0 $&$ 0$ \\
				$-e^{-\alpha _1} $&$ 0 $&$ 0 $&$ 0 $&$ 0 $&$ 0$ \\
				$0 $&$ 0 $&$ 0 $&$ 0 $&$ -e^{\alpha _1+\alpha _2} $&$ 0$ \\
				$0 $&$ 0 $&$ 0 $&$ 0 $&$ 0 $&$ -e^{\alpha _2}$ \\
				$0 $&$ 0 $&$ -e^{-\alpha _1-\alpha _2} $&$ 0 $&$ 0 $&$ 0$ \\
				$0 $&$ 0 $&$ 0 $&$ -e^{-\alpha _2} $&$ 0 $&$ 0$ 
			}%
		\end{tabular}
		\caption{Matrix of automorphism $s_1$ in the $\oooo$ basis for the flag variety $SL_3/B$.}
	}
\end{figure}
\section{Conclusions} \label{Conclusions}
In order to calculate the coefficients of the automorphism $s_i^K$, we must multiply the class of Schubert variety with the class of a linear bundle twice, which means using the Chevalley formula twice, which is computationally challenging. Switching to the motivic Chern classes basis changes the Chevalley formula (shown in \cite{MNS}), but the formula is still combinatorically too complicated to simplify these calculations. As we saw in Chapter \ref{Chevalley formula}, the problem with the Chevalley formula is that the support of it is not described in terms of the elements of the Weyl group, but the monoid $W'$. We only know that the support is subset of $\{\oo[\omega']\}$ for $\omega'$ being the subword of $\omega s_i$. That is because, for example, a certain subword can be constructed in two ways in the Weyl group, but in three ways in the monoid $W'$, as we explain below. \\ \\ Let us take a word $s_1 s_2 s_1$ in the Weyl group of $SL_n$. It has a subword $s_1$. It can be constructed by removing $s_2$ and the last $s_1$ or by removing $s_2$ and the first $s_1$. However, in this monoid a relation $s_1 s_1 = s_1$ holds, so there are three constructions possible - the two mentioned above and one additional, executed by only removing $s_2$. These additional combinations can transform element from being in the support to not being in the support and the other way around. For example the coefficient of $\oo[\id]$ in $s_1^K \oo[s_1 s_2]$ is equal to zero (Figure \ref{Figure - action s_1 in SL_3}), but would be non-zero if we would consider only subwords in the Weyl group because one non-zero term of $\llll[-\alpha_1] \oo[s_1 s_2 s_1]$ is equal to
\begin{equation*}
	T_{\alpha_1}^1 T_{\alpha_2}^0  T_{\alpha_1}^0(e^{-\alpha_1}).
\end{equation*}
A possible solution to this problem might be a different choice of bases or indexing elements of basis $\mathcal{B}$ not by elements of the Weyl group $W$, but by elements of the monoid (or, in fact, the distinguished basis of Hecke algebra) $W'$. These two objects have the same elements, but the fact that the relations holding in the monoid are different from the Weyl group could help to describe the support of the automorphism $s_i^K$ explicitely.
\printbibliography

\end{document}